\documentclass[12pt]{amsart}
\usepackage[latin9]{inputenc}
\usepackage{amsthm}
\usepackage{amstext}
\usepackage{amssymb}
\usepackage{tikz}
\makeatletter
\numberwithin{equation}{section}
\numberwithin{figure}{section}

\usepackage{mathrsfs}
\usepackage{amssymb}
\usepackage{amsthm}
\usepackage{dsfont}
\usepackage{amssymb}
\usepackage{amsmath,amscd}
\usepackage[all,cmtip]{xy}
\usepackage{mathrsfs}
\usepackage{multirow}
\usepackage{hyperref}
\usepackage{graphicx}
\usepackage{amsfonts, amsthm}
\usepackage{fullpage}
\usepackage{enumerate}
\usepackage{appendix}
\usepackage{cite}

\normalfont\upshape

\theoremstyle{plain}
\newtheorem{thm}[subsection]{Theorem}
\newtheorem{lemma}[subsection]{Lemma}
\newtheorem{prop}[subsection]{Proposition}

\theoremstyle{definition}
\newtheorem{rmk}[subsection]{Remark}
\newtheorem{defn}[subsection]{Definition}

\let\l\lambda

\let\r\rho
\let\s\sigma

\let\th\theta

\def\scr{\mathscr}
\def\cal{\mathcal}

\let\G\Gamma

\newcommand{\ra}{\longrightarrow}

\newcommand{\ro}{{\mbox{ or }}}

\newcommand{\im}{{\rm im}\:}

\newcommand{\GG}{{\mathbb G}}
\newcommand{\F}{{\mathbb F}}
\newcommand{\Z}{{\mathbb Z}}
\newcommand{\C}{{\mathbb C}}
\newcommand{\Q}{{\mathbb Q}}
\newcommand{\R}{{\mathbb R}}

\newcommand{\cA}{\mathcal{A}}

\newcommand{\cF}{\mathcal{F}}

\newcommand{\cE}{\mathcal{E}}

\newcommand{\cO}{\mathcal{O}}

\newcommand{\geom}{\mathrm{geom}}

\newcommand{\End}{{\mathrm{End}}}

\newcommand{\cri}{{\mathrm{cris}}}

\newcommand{\Aut}{{\mathrm{Aut}}}
\newcommand{\et}{{\mathrm{et}}}

\makeatother

\begin{document}

\title{$l$-adic monodromy and Shimura curves in positive characteristics}

\author{Jie Xia}

\address{Mathematics Department, Columbia University in the city of New York}

\email{xiajie@math.columbia.edu}

\maketitle
\begin{abstract}
In this paper, we seek an appropriate definition for a Shimura curve of Hodge type in positive characteristics, i.e. a characterization, in terms of geometry mod $p$, of curves in positive characteristics which are reduction of Shimura curve over $\C$. Specifically, we study the liftablity of a curve in moduli space $\cal{A}_{4,1,k}$ of principally polarized abelian varieties over $k, \text{char } k=p$. We show that some conditions on the $l$-adic monodromy over such a curve imply that this curve can be lifted to a Shimura curve. 
\end{abstract}
\section{Introduction}

This paper is a sequel to our previous paper \cite{Xia4} and \cite{Xia3}. All the three paper aim at answering the following question: 
\[\textit{What is an appropriate definition of Shimura curves in positive characteristics ?}
\]
The answer is only known for Shimura varieties of PEL type which admit a natural moduli interpretation. In this paper, we consider Shimura curves of Hodge type and give an answer in the generic ordinary
case, in terms of $l$-adic monodromy.

In \cite{Xia4}, we start with a proper family of abelian varieties in characteristic $p$ and prove if this family admits some
special crystalline cycles, then it is a reduction of a Shimura curve
of Mumford type. In this paper we present a similar result. We find the conditions on the $l$-adic monodromy associated to the family, which imply the family is a reduction of a Mumford type family of abelian fourfolds. 


Let $\pi:X\ra C$ be a family of principally polarized abelian fourfolds over a proper smooth
curve $C$ defined over a finite field ${\F}_{q}$ with $q=p^{f}$,
$p>3$. 
Let $\cE_{l}$ be the lisse \'etale $l$-adic sheaf $R^{1}\pi_{*}(\Q_{l})$
over $C$. Choose a geometric point $\bar{\xi}$ and a closed point $c$ in $C$ , and then $\cE_{l}$
induces a monodromy: 
\[
\r:\pi_{1}(C,\bar{\xi})\ra\Aut(\cE_{l,c})\cong GL(8,\Q_{l}).
\]
 Let $G_{l}$ be the Zariski closure of $\r(\pi_{1}(C,\bar{\xi}))$
in $GL(8,\Q_{l})$ and $G_{l}^{\geom}$ be that of $\r(\pi_{1}^{\geom}(C,\bar{\xi}))$.
\[
1\ra\pi_{1}^{\geom}(C,\bar{\xi})\ra \pi_{1}(C,\bar{\xi})\ra\mbox{Gal}(\bar{\F}_q/\F_{q})\ra1.
\]Then $G_{l}^{\geom}$ is a normal subgroup of $G_{l}$.

To every closed point $c\in C$, it associates a unique (up to conjugation) Frobenius element $F_{c}$ in $\pi_{1}(C,\bar{\xi})$. Its image in
$\mbox{Gal}(\bar{\F}_q/\F_{q})$ is the integer $\deg(\kappa(c):\F_{q})$.

For notational simplicity, let us define the following representation
\[
\r_{0}:SL(2,\Q_{l})^{\times3}\ra \Aut(\cE_{l,c})
\]
 as the tensor product of three copies of the standard representation of $SL(2,\Q_{l})$. 

Fix an embedding $\Q_{l}\ra\C$ once for all. Our main theorem is
as follows. \begin{thm}\label{main theorem} If there exists a
closed point $c\in C$ such that 
\begin{enumerate}
\item $G_{l}^{\geom,o}\otimes_{\Q_{l}}\C\cong\im\r_{0}\otimes\C$, 
\item $X_{c}$ is an ordinary abelian variety, 
\item In $G_{l}$, $\r(F_{c})$ generates a maximal torus which is unramified
over $\Q_{p}$. 
\end{enumerate}
Then $X\ra C$ is a weak Mumford curve. 

If further assume the Higgs
field of $X\ra C$ is maximal, then there exists a family of polarized abelian fourfolds $Y\ra C'$ such that 

\begin{enumerate}
\item $C'\ra C$ is a finite \'etale covering, 
\item $Y\ra X'$ is an isogeny between $Y$ and the pullback
family of $X$ over $C'$,
\item $Y\ra C'$ is a good reduction of a Mumford curve. 
\end{enumerate}
\end{thm}
For the definition of weak Mumford curve, see Section \ref{Mumford curves}.
\begin{rmk} By \ref{Frobeius torus} and Chebotarev density theorem, the Frobenius element over sufficiently many points in $C$ generates the maximal torus in $G_l$. From \cite{Larsen}, we know that the torus generated
by the Frobenius $F_{c}$ is defined over $\Q$. So it makes sense
to require the torus is unramified over $\Q_{p}$, which is equivalent
to say that the eigenvalues of the Frobenius are unramified over $\Q_{p}$.  

We wonder if (3) can be replaced by a weaker condition, especially under the presence of (1) and (2). More are explained in Remark \ref{unramified condition}.
 \end{rmk} 


\subsection{Structure of the paper}

To prove \ref{main theorem}, we reduce it to the main theorem in \cite{Xia4}. We need
to compute the Frobenius eigenvalues in the $l$-adic setting and the
connection with the crystalline cohomology. 

In Section \ref{examples}, we show some good reductions of Mumford
curves satisfy the conditions in \ref{main theorem} so that we are not proving a vacuous theorem.

In Section \ref{Frobenius eigenvalues}, we compute the dimensions of the Frobenius eigenspaces in
$H_{\et}^{0}(C_{\bar{\F}_{q}},\wedge^{4}\cE_{l})$ and $H_{\et}^{0}(C_{\bar{\F}_{q}},\scr{E} nd(\wedge^{2}\cE_{l}))$.
\begin{equation}
\begin{aligned}\dim H_{\et}^{0}(C_{\bar{\F}_{q}},\wedge^{4}\cE_{l})^{F-q^{2}}\otimes\Q_{q} & =1\\
\dim H_{\et}^{0}(C_{\bar{\F}_{q}},\scr{E} nd(\wedge^{2}\cE_{l}))^{F}\otimes\Q_{q} & =4.
\end{aligned}
\end{equation}

 In Section \ref{comparison}, we translate the above result to $p$-adic
case via comparing Lefschetz trace formulas. Therefore the Frobenius eigenspaces in crystalline cohomology spaces have the expected dimensions as in the main theorem of \cite{Xia4}. Furthermore, we conclude in Section \ref{compute} that 
\[
\G((C/W(k))_{\cri},\scr{E}nd(\wedge^{2}\cE))^{F}\otimes\Q_{q} \cong \Q_{q}^{\times4}
\]  as algebras. Then \ref{main theorem} boils down
to the main result in \cite{Xia4}.

\textbf{Acknowledgements.} I am deeply indebted to my thesis advisor Aise Johan de Jong, who introduced me to this subject and encouraged me through this project. Without him I would have given up long time ago. I thank Martin Olsson for useful discussion and feedbacks. 
\section{Mumford curves}

\label{Mumford curves} 
In \cite[Chapter 4]{Mum}, Mumford defines a family of Shimura curves.
We briefly recall the construction. 

Let $F$ be a cubic
totally real field and $D$ be a quaternion division algebra over
$F$ such that 
\[
D\otimes_{\Q}\R\cong\mathbb{H}\times\mathbb{H}\times M_{2}(\R),{\rm {Cor}_{F/\Q}(D)\cong M_{8}(\Q).}
\]
 Here $\mathbb{H}$ is the quaternion algebra.

Let $G=\{x\in D^{*}|x\bar{x}=1\}$. Then $G$ is a $\Q$-simple algebraic
group and it is the $\Q$-form of the $\R$-algebraic group $SU(2)\times SU(2)\times SL(2,\R)$.
\begin{align*}
h:\mathbb{S}_{m}(\R)\ra & G(\R)\\
e^{i\th}\mapsto & I_{4}\otimes\begin{pmatrix}\cos\th & \sin\th\\
-\sin\th & \cos\th
\end{pmatrix}.
\end{align*}

The pair $(G,h)$ forms a Shimura datum. And it defines Shimura
curves, parameterizing abelian fourfolds. 
We call such curves (with its universal family) \textit{Mumford curves}.

\begin{defn}

\begin{enumerate}
\item A curve in $\cA_{4,1,n}\otimes\C$ is called a \textit{special
Mumford curve} if it is the image of a Mumford curve in $\cA_{4,1,n}\otimes\C$
induced by a universal family.
\item The family $X\ra C$ is a \textit{weak Mumford curve over $k$} if
the image of $C\ra\cA_{4,1,n}$ (induced by the family $X/C$)
is (possible an irreducible component of) a reduction of a special
Mumford curve in $\cA_{4,1,n}\otimes\C$ .
\end{enumerate}
\end{defn} 

\begin{rmk}
The "weakness" of the weak Mumford curve,  comparing to good reductions, reflects in two aspects: firstly, the image of $C\ra \cA_{4,1,n}$ might have singularities. Secondly, the reduction of a special Mumford curve at $k$ might be reducible and that image of $C\ra \cA_{4,1,n}$ is just one of the irreducible components. 
\end{rmk}

\section{Examples}

\label{examples} To indicate that \ref{main theorem} is not a vacuous
result, we show good reductions of a Mumford curve with an ordinary
fiber satisfy the conditions of \ref{main theorem}.

Let $f: A\ra M$ be the universal family over a Mumford curve defined
over $\C$ and $M$ is defined over the reflex field $K$. For every
$p$ over which $K$ splits, $A\ra M$ admits a smooth and generically
ordinary reduction $\pi: X\ra C$ over $p$ (\cite{Zuo2}). By the definition
of Mumford curve or \cite[Proposition 2.4]{Xia2}, the image of 
\[
\r_{\C}:\pi_{1}(M)\ra R^{1}f_{*}(\underline{\C})
\]
 is $SL(2,\C)^{\times3}$. By Grothendieck specialization theorem of algebraic
monodromy, $\r_{\C}$ factors through 
\[
\r:\pi_{1}(C,\bar{\xi})\ra R^{1}f_{*}(\underline{\C})
\]
 with a surjection $\pi_{1}(M)\ra\pi_{1}(\bar{C},\bar{\xi})$. By
the comparison of the $l$-adic cohomology and de Rham cohomology
$R^{1}f_{*}(\C)\cong R^{1}f_{*}(\Q_{l})\otimes\C$ for every $l\neq p$,
we know conditions (1) and (2) in \ref{main theorem} hold for such $C$. 

For condition (3), we choose $x$ to be a CM point on $M$. Since
$A_{x}$ is simple, there is a degree 8 field $L\subset\End^{o}(A_{x})$.
We can choose $p$ such that $L$ is unramified over $p$. Then look
at the reduction $\bar{x}$ and $\Q[F_{\bar{x}}]$ is the center of
$\End(X_{\bar{x}})$. Since $L\subset\End(X_{\bar{x}})$ is the maximal
commutative subalgebra, $F\in L$ and in particular, $F$ is unramified
over $p$.

So with a careful choice of $p$, the resultant reduction of a Mumford
curve satisfies all the conditions in \ref{main theorem}.

\section{Frobenius eigenvalues on $\wedge^{4}\cE\ro{\cal {E}nd(\wedge^{2}\cE)}$}

\label{Frobenius eigenvalues} In this section, we compute the dimension
of Frobenius eigenspace at the target spaces and the corresponding
eigenvalues.

Recall the short exact sequence 
\[
1\ra\pi_{1}^{\geom}(C,\bar{\xi})\ra\pi_{1}(C,\bar{\xi})\ra{\rm {Gal}(\bar{\F}_{q}/\F_{q})\ra1.}
\]
The Zariski closure $G^{\geom}_l$, of $\r(\pi^{\geom}_1(C,\bar{\xi}))$ is a normal subgroup of $G_l$. Since $\cE_l$ is pure, it follows from \cite[1.3.9,3.4.1(iii)]{Weil II} that $G^{\geom}_l$ is semisimple. The connected component of identity $G^{\geom,o}_l$ is the derived group of $G^o_l$.


\begin{rmk}
If we assume $G^{\geom}_l \otimes \C$ is entirely contained in $SL(2,\C)^{\times 3}$, then we only need to enlarge the base field $\F_q$ to kill the $S_3$ part. 
\end{rmk}

 For every closed point $c$ of $C$, the action of $F_{c}$ on $\cE_{l\, c}\cong H_{\et}^{1}(X_{\bar{c}},\Q_{l})$
is semisimple (see \cite{Chi}). Therefore $F_{c}$ acts on $\wedge^{4}\cE_{l,c}$
and $\End(\wedge^{2}\cE_{l,c})$ semisimply.


Firstly let us consider the space 
\[
H_{\et}^{0}(C_{\bar{\F}_{q}},\scr{E} nd(\wedge^{2}\cE_{l}))\cong\End(\wedge^{2}(\cE_{l,c}))^{\pi_{1}^{\geom}(C,\bar{c})}.
\]
 Since $F_c$ acts on $\End(\wedge^{2}\cE_{l,c})$ semisimply, we can calculate
its eigenvalues over $\C$. 

Let $V$ be the dimension 2 standard representation of $SL(2,\C)$. Condition (1) of Theorem \ref{main theorem} shows that $\cE_{l.c}\otimes \C \cong V^{\otimes3} $ is the tensor product of three standard representations of $SL(2,\C)$.   So base change to $\C$, 
\[
H_{\et}^{0}(C_{\bar{\F}_{q}},\scr{E} nd(\wedge^{2}\cE_{l}))\otimes\C\cong\End(\wedge^{2}(V^{\otimes3}))^{SL(2,\C)^{\times3}}.
\]

As a representation of $SL(2,\C)^{\times3}$, $\wedge^{2}(V^{\otimes3})$
decomposes into four distinct irreducible components 
\[
\wedge^{2}(V^{\otimes3})\cong\oplus^4_{i=1}W_{i}.
\] 
There $W_1, W_2, W_3$ are all $S^2 V \otimes S^2 V$ and $W_4$ is the trivial representation of $SL(2,\C)^{\times 3}$. Yet $W_1, W_2, W_3$ are pairwisely non-isomorphic $SL(2, \C)^{\times 3}$ representations. 

Let $p_k$ ($i_k$, resp.) be the projection from $\wedge^{2}(V^{\otimes3})$ to $W_k$ (inclusion from $W_k$ to $\wedge^{2}(V^{\otimes3})$, resp.). We have \[\End(\wedge^{2}(V^{\otimes3}))^{SL(2,\C)^{\times3}}\cong \oplus^4_{k=1} (i_k\circ \text{id}_{W_k}\circ p_k).\] 
Note each id$_{W_k}$ is invariant under the action of $SL(2,\C)^{\times 3}$ and the scalar multiplication. Thus id$_{W_k}$ is further invariant under the action of $G_l$. In particular, the Frobenius $F_c$ fixes each $\text{id}_{V_{i}}$. So the invariant
space 
$
H_{\et}^{0}(C_{\bar{\F}_{q}},\scr{E} nd(\wedge^{2}\cE_{l}))^{F}
$ has dimension 4.

Secondly, 
\[
H_{\et}^{0}(C_{\bar{\F}_{q}},\wedge^{4}\cE_{l})=(\wedge^{4}\cE_{l\,{c}})^{\pi_{1}^{\geom}(C,\bar{\xi})}.
\]
Similarly, base change to $\C$ and it is isomorphic to $\wedge^{4}(V^{\otimes3})^{SL(2,\C)^{\times3}}$.
One can directly compute by hand, or see the proof of Theorem
4.1 in \cite{Moon} to conclude that this space only has dimension
1 which is generated by the polarization. Therefore the corresponding Frobenius
eigenvalues are $q^{2}$.

In summary, we have the following results: 
\begin{equation}
\begin{aligned}\dim H_{\et}^{0}(C_{\bar{\F}_{q}},\scr{E} nd(\wedge^{2}\cE_{l}))^{F} & =4,\\
\dim H_{\et}^{0}(C_{\bar{\F}_{q}},\wedge^{4}\cE_{l})^{F-q^{2}} & =1.
\end{aligned}
\label{l-adic results}
\end{equation}

\section{Comparison of Lefschetz Trace Formulas}

\label{comparison} In this section, we compare Lefschetz Trace Formulas
to obtain a similar result to (\ref{l-adic results}) in the case of crystalline cohomology. 

We firstly consider
$\cE_{p}:=R^{1}\pi_{\cri,*}(\cO_{X})$. Since $\s$ is the identity
on $\F_{q}$, the absolute Frobenius $F$ acts linearly on $\cE_{p,c}$.
Since the local crystalline characteristic polynomial coincides with
the $l$-adic one (\cite[1.3.5]{Ill3}) 
\begin{equation}\label{Ill}
\det(1-tF|_{\cE_{p\, c}})=\det(1-tF|_{\cE_{l,c}}),
\end{equation}
 the eigenvalues of $F$ on $\cE_{l,c}$ and $\cE_{p\, c}$ are
identical.

Let $\cF_{l}$ be either $\wedge^{4}\cE_{l}\ro\scr{E} nd(\wedge^{2}\cE_{l})$.
Since $\cE_{l}$ comes from geometry, by Deligne's Weil II, the $l$-adic
relative Lefschetz Trace Formula provides 
\[
\prod_{c\in C}\det(1-tF|_{\cF_{l,c}})=\prod_{i}\det(1-tF|_{H_{\et}^{i}(C_{\bar{\F}_{q}},\cF_{l})})^{(-1)^{i}}.
\]

In the $p$-adic case, we still use $\cF_{p}$ to represent either
$\wedge^{4}\cE_{p}\ro\scr{E} nd(\wedge^{2}\cE_{p})$. Since $\cE_{p}$
is a Dieudonne crystal, $\cF_{p}$ is automatically overconvergent.
By a theorem of Etesse and le Stum (\cite[2.1.2]{Kedlaya}), we also
have a Lefschetz Trace Formula within crystalline cohomology setting
\[
\prod_{c\in C}\det(1-tF|_{\cF_{p,c}})=\prod_{i}\det(1-tF|_{H^{i}_\cri(C/{\Z_{q}},\cF_{p})})^{(-1)^{i}}.
\]
 Combining with equality (\ref{Ill}), we have 
\begin{equation}
\prod_{i}\det(1-tF|_{H_{\et}^{i}(C_{\bar{\F}_{q}},\cF_{l})})^{(-1)^{i}}=\prod_{i}\det(1-tF|_{H^{i}_\cri(C/\Z_{q},\cF_{p})})^{(-1)^{i}}.\label{the formula}
\end{equation}

By Deligne's Weil II \cite{WeilII}, the \'etale cohomology groups $H_{\et}^{i}(C_{\bar{\F}_{q}},\cF_{l})$
is pure of weight $i+j$ where $\cF_{l}$ has weight $j$. Since $\cF_{p}$
is pointwisely pure, by \cite[Theorem 5.3.2]{Kedlaya}, $H_{\cri}^{i}(C/\Z_{q},\cF_{p})$
has the purity which implies, on each side of equality (\ref{the formula}),
there is no cancellation between the numerator and the denominator. All zeros or poles have the expected
complex norms. Then we have the following termwise equality from (\ref{the formula}).
\begin{equation}
\det(1-tF|_{H_{\et}^{i}(C_{\bar{\F}_{q}},\cF_{l})})=\det(1-tF|_{H_{\cri}^{i}(C/\Z_{q},\cF_{p})}).\label{the formula 2}
\end{equation}
 So the eigenvalues of $F$ on $H^{0}(C/\Z_{q},\wedge^{4}\cE_{p})\otimes\Q_{q}$
and $H^{0}(C/\Z_{q},\scr{E} nd(\wedge^{2}\cE_{p}))\otimes\Q_{q}$ are
identical as on their $l$-adic counterparts. In particular, 
\begin{equation}
\begin{aligned}\dim H^{0}(C/\Z_{q},\wedge^{4}\cE_{p})^{F-q^{2}}\otimes\Q_{q} & =1,\\
\dim_{\Q_{q}}H^{0}(C/\Z_{q},\scr{E} nd(\wedge^{2}\cE_{p}))^{F}\otimes\Q_{q} & =4.
\end{aligned}
\label{p-adic result}
\end{equation}

\section{Compute $\End^{0}(\wedge^{2}\cE_{p})^{F}$}

\label{compute}

In order to apply the main theorem in \cite{Xia4}, we need to prove that $\End^{0}(\wedge^{2}\cE_{p})^{F}\cong\Q_{q}^{\times4}$
as algebras.

\subsection{Frobenius Torus}

\label{Frobeius torus} By \cite[Theorem 2]{Tate}, $\Q[F]\cong\prod K_{i}$
where $K_{i}$ are number fields. The multiplicative group $\Q[F]^{*}$
defines a $\Q$-torus 
\[
T=\prod_{i}{\rm {Res}_{K_{i}/\Q}(\GG_{m}).}
\]
 Viewing $F$ as an element in $G_{l}$, $T$ can be regarded as the
$\Q$-model of the connected component of 1 in the Zariski closure
of the set $\{\r(F)^{n}|n\in\Z\}$ in $G_{l}$ (cf. \cite[Chapter II, Section 13, Proposition 3]{Chevalley}).
In particular, $T$ is contained in a maximal torus of $G_{l}$.

By \cite[Theorem 3.7]{Chi 2} and Chebotarev density theorem for the function field, 
generic points $c$ on $C$ satisfy that $F_{c}$ generates a maximal torus. 
For every $c$, the torus $T$ is defined over $\mathbb{Q}$. We say $T$ is \textit{unramified} over $\mathbb{Q}_{p}$ if the splitting field of $T$ is unramified over prime $p$, and equivalently, the eigenvalues of $F_{c}$ are unramified over $p$. 

\begin{rmk} \label{unramified condition}
Varying the prime $l$, we obtain a compatible system of $l$-adic representation as stated in \cite[6.5]{Larsen}. The existence of a point $c$ satisfying (3) in \ref{main theorem} requires that $G_l$ is unramified over $\Q_p$. 

On one hand, by \cite[Proposition 8.9]{Larsen} and \cite[Proposition 1.2, Theorem 3.2]{Larsen2}, for a subset of primes $l$ of density $1$(or even $l$ large enough), $G_l$ is unramified over $Q_l$. However, most results in the two paper have involved Dirichlet density restriction and hence can not be applied directly to our case.

On the other hand, we expect that if $G_l$ is  unramified over $\Q_q$, then there always exists a closed point $c$ satisfying (3) in \ref{main theorem}. 

We also think generic ordinary property of $X\ra C$ should also provide more information on Frobenius eigenvalues. 
\end{rmk}

\subsection{Eigenvalues of $F_c$ on $\cE_{p\,c}$}
\label{l}
Note up to now, we have not used condition (2) and (3) in \ref{main theorem}. Under the condition (2), by \ref{Frobeius torus}, we always can find $c$ such that $X_c$ ordinary and $\r(F_c)$ a maximal torus. 
Further with the condition (3), there exists a closed point $c$ which satisfies the following two conditions: 
\begin{enumerate}
\item $X_{c}$ is ordinary, 
\item the Frobenius torus $T$ is a maximal torus in $G_{l}$. 
\end{enumerate}

Now we study the eigenvalues of the Frobenius on the fiber over $c$.
Since $X_{c}$ is ordinary, $\cE_{p\, c}$ is the product of a unit
root crystal ${\cal {U}_c}$ and its dual ${\cal {U}^{\vee}_c}$. Let
$\l_{1},\cdots,\l_{4}$ be the eigenvalues of $F_c$ on ${\cal {U}_c}$.
Then on ${\cal {U}^{\vee}_c}$, the eigenvalues of $F_c$ are $\displaystyle\frac{q}{\l_{i}}$.
Since ${\cal {U}_c}$ is a unit root crystal, $\l_{i}$ are all $p$-adic units.
Since $\cE_{p\, c}$ has pure weight 1, $\l_{i}$ all have complex norm $q^{\frac{1}{2}}$. 

By \ref{Ill}, the Frobenius $F_c$ also has eigenvalues $\l_1,\cdots, \l_4, \displaystyle\frac{q}{\l_1}, \cdots, \displaystyle\frac{q}{\l_4}$ on $\cE_{l,c}$. Since the Frobenius torus is the maximal torus, the Frobenius eigenvalues $\l_i$ correspond to the weights in the $SL(2)^{\times 3}$ representation $V^{\otimes 3}$. Let $a, b, c$ be the three highest weights in the three standard representation of $SL(2, \C)$. Then the eight weights of $V^{\otimes 3}$ are of the form $\pm a \pm b \pm c$ and they have a configuration as vertices of a cube. In this cube, the four $p$-adic units $\l_1, \cdots, \l_4$ lie in the same face.  Without loss of generality, we can assume that $\l_1$ corresponds to the highest weight $a+b+c$ and $\l_2, \l_3, \l_4$ correspond to $a+b-c, a+c-b$ and $a-b-c$. Then the only relation between $\l_1, \cdots,\l_4$ is $\l_1\l_4=\l_2\l_3$. So we have the following lemma.

\begin{lemma} \label{no relation}
Under the above choice of $c$, the eigenvalues $\l_i$ have no relations other than those generated by  $\l_i\displaystyle\frac{q}{\l_i}=q $ and $\l_1\l_4=\l_2\l_3$.
\end{lemma}
\begin{rmk}
Lemma \ref{no relation} also follows from the arguments in \cite[Section 4]{Noot}. 
\end{rmk}


\begin{prop} 
\[
\End^{0}(\wedge^{2}\cE_{p})^{F}\cong\Q_{q}^{\times4}
\]
 as algebras. \end{prop} \begin{proof}

From \cite[Proposition 5.15]{Xia4} or basic representation theory of $SL(2)$,
we know the condition \ref{p-adic result} implies $\End(\wedge^{2}\cE_{p})^{F}\otimes_{\Q_{q}}\C\cong\C^
{\times 4}$
as algebras. In particular, the algebra $\End(\wedge^{2}\cE_{p})^F$ is commutative.
Therefore $\End^{0}(\wedge^{2}\cE_{p})^F$ is a product of fields.

Note $\wedge^{2}\cE_{p}$ has the polarization as a direct summand.
So 
\[\End(\wedge^{2}\cE_{p})^{F}\cong\Q^{\times 4}_q, \Q_{q}\times K\ro\Q_{q}^{\times2}\times L\]
where $K$ is a degree 3 field extension of $\Q_{q}$ and $L$ has degree
2. Comparing with the decomposition over $\C$, there exists $\eta_K \in K\ro \eta_L \in L$ such that $\im \eta_K $ and $\im \eta_L$ are subcrystals in $\wedge^2 \cE_p$. Further, rank $\im \eta_K=27$ and rank $\im \eta_L=18$.

If $K\ro L$
is unramified over $\Q_{q}$, then by enlarging $f$ in $q=p^{f}$,
it becomes a product of copies of $\Q_{q}$. Therefore we only need
to consider the case $K\ro L$ ramified over $\Q_{q}$. Since $p\neq2\ro 3$, we can
assume $L\cong\Q_{q}(\sqrt{p})$ and $K\cong\Q_{q}(\sqrt[3]{p})$ and we can choose $\eta_K=\sqrt[3]{p}$, $\eta_L=\sqrt{p}$. 

Note $\cE_{p\, c}\cong \cal{U}_c\oplus \cal{U}^\vee_c$. Since the eigenvalues have distinct $p$-adic values, there is no $F_c$-invariant morphisms between ${\wedge^{2}{\cal {U}_{c}}}, \wedge^{2}\cal {U}_{c}^{\vee}$ and $\cal {U}_{c}\otimes\cal {U}_{c}^{\vee}$. Thus we have the decomposition
\[
\End(\wedge^{2}\cE_{p})^F \ra\End(\wedge^{2}\cE_{p\, c})^F \cong\End({\wedge^{2}{\cal {U}_{c}}})\oplus\End(\wedge^{2}{\cal {U}_{c}^{\vee})\oplus\End({\cal {U}_{c}\otimes{\cal {U}_{c}^{\vee}).}}}
\]
The restriction of $F$ to $\End(\wedge^{2}\cE_{p\, c})$ is just as $F_c$. Then by \ref{no relation}, all the eigenvalues of $F$ on $\wedge^{2}{\cal {U}_{c}}$
are $\l_{1}\l_{2},\cdots,\l_{3}\l_{4}$ and there is no more relations between the eigenvalues
of $\wedge^{2}{\cal {U}_{c}}$ other than $\l_{4}\l_{1}=\l_{2}\l_{3}$. So each eigenspace $U_{\l_i\l_j}$ has dimension 1 except for $(1,4)\ro (2,3)$.
Thereby 
\begin{multline}
\End(\wedge^{2}{\cal {U}_{c}})^F \cong\oplus_{(i,j)\neq(1,3),(2,4)}\End(U_{\l_{i}\l_{j}})\oplus\End
(U_{\l_{1}\l_{3}})\\ \cong\oplus_{(i,j)\neq(1,3),(2,4)}\Q_{q}(\l_{i}\l_{j})\oplus M_2(\Q_q(\l_1\l_4)).
\end{multline}
Since the four eigenvalues
$\l_{i}$ are all unramified over $\Q_{q}$ and $L \ro K$ is ramified, the image of the composition
\[
L\ro K \ra\End(\wedge^2\cE_p)^F \ra \End(\wedge^2\cal{U}_c)^F
\]
lies only in $\End(U_{\l_1\l_4})\cong M_2(\Q_q(\l_1\l_4))$. Otherwise, it would induce an embedding $L\ro K \hookrightarrow \Q_q(\l_i\l_j)$.
In particular, ${\eta_K}_{|\wedge^2 \cal{U}_c} \ro {\eta_L}_{|\wedge^2 \cal{U}_c}$ has only rank 2.

Restricted to point $c$, the image of $\eta_K$ has dimension at most
only 20. Contradiction. 

For $L$, we know that ${\eta_L}_{|\cal{U}_c \otimes \cal{U}^\vee_c}$ is a surjection.
Note the eigenvalues of $F_c$ on $\cal{U}_c \otimes \cal{U}^\vee_c$ have the form $\displaystyle\frac{q\l_i}{\l_j}$. Again by \ref{no relation}, among these eigenvalues, $\displaystyle\frac{q\l_1}{\l_4}$ has only multiplicity 1. Therefore 
\[\End(\cal{U}_c \otimes \cal{U}^\vee_c)^F \cong \End(\cal{U}_{\frac{q\l_1}{\l_4}} )\oplus \cdots \cong \Q_q(\frac{q\l_1}{\l_4}) \oplus \cdots \] as algebras. Since $\Q_q(\frac{q\l_1}{\l_4})$ is unramified over $\Q_q$, the image of $L$ in $\End(\cal{U}_c \otimes \cal{U}^\vee_c)^F$ excludes $\End(\cal{U}_{\frac{q\l_1}{\l_4}} )$ and hence $\eta_L$ can not be a surjection. The contradiction concludes the proof.
\end{proof}


\bibliographystyle{amsplain}

\bibliography{mybib}

\providecommand{\bysame}{\leavevmode\hbox to3em{\hrulefill}\thinspace}
\providecommand{\MR}{\relax\ifhmode\unskip\space\fi MR }
\providecommand{\MRhref}[2]{%
  \href{http://www.ams.org/mathscinet-getitem?mr=#1}{#2}
}
\providecommand{\href}[2]{#2}
\begin{thebibliography}{10}

\bibitem{Chevalley}
Claude Chevalley, \emph{Th\'eorie des groupes de {L}ie. {T}ome {II}. {G}roupes
  alg\'ebriques}, Actualit\'es Sci. Ind. no. 1152, Hermann \& Cie., Paris,
  1951. \MR{0051242 (14,448d)}

\bibitem{Chi}
W{\^e}n~Ch{\^e}n Chi, \emph{On the {$l$}-adic representations attached to some
  absolutely simple abelian varieties of type {${\rm II}$}}, J. Fac. Sci. Univ.
  Tokyo Sect. IA Math. \textbf{37} (1990), no.~2, 467--484. \MR{1071431
  (91g:11060)}

\bibitem{Chi2}
\bysame, \emph{{$l$}-adic and {$\lambda$}-adic representations associated to
  abelian varieties defined over number fields}, Amer. J. Math. \textbf{114}
  (1992), no.~2, 315--353. \MR{1156568 (93c:11038)}

\bibitem{WeilII}
Pierre Deligne, \emph{La conjecture de {W}eil. {II}}, Inst. Hautes \'Etudes
  Sci. Publ. Math. (1980), no.~52, 137--252. \MR{601520 (83c:14017)}

\bibitem{Ill3}
Luc Illusie, \emph{Crystalline cohomology}, Motives ({S}eattle, {WA}, 1991),
  Proc. Sympos. Pure Math., vol.~55, Amer. Math. Soc., Providence, RI, 1994,
  pp.~43--70. \MR{1265522 (95a:14021)}

\bibitem{Kedlaya}
Kiran~S. Kedlaya, \emph{Fourier transforms and {$p$}-adic `{W}eil {II}'},
  Compos. Math. \textbf{142} (2006), no.~6, 1426--1450. \MR{2278753
  (2008b:14024)}

\bibitem{Larsen}
M.~Larsen and R.~Pink, \emph{On {$l$}-independence of algebraic monodromy
  groups in compatible systems of representations}, Invent. Math. \textbf{107}
  (1992), no.~3, 603--636. \MR{1150604 (93h:22031)}

\bibitem{Larsen2}
\bysame, \emph{Abelian varieties, {$l$}-adic representations, and
  {$l$}-independence}, Math. Ann. \textbf{302} (1995), no.~3, 561--579.
  \MR{1339927 (97e:14057)}

\bibitem{Zuo2}
Sheng Mao and Kang Zuo, \emph{On the newton polygons of abelian varieties of
  mumford's type}, arXiv:1106.3505v1.

\bibitem{Moon}
B.~J.~J. Moonen and Yu.~G. Zarhin, \emph{Hodge classes and {T}ate classes on
  simple abelian fourfolds}, Duke Math. J. \textbf{77} (1995), no.~3, 553--581.
  \MR{1324634 (96b:14010)}

\bibitem{Mum}
D.~Mumford, \emph{A note of {S}himura's paper ``{D}iscontinuous groups and
  abelian varieties''}, Math. Ann. \textbf{181} (1969), 345--351. \MR{0248146
  (40 \#1400)}

\bibitem{Noot}
Rutger Noot, \emph{Abelian varieties with {$l$}-adic {G}alois representation of
  {M}umford's type}, J. Reine Angew. Math. \textbf{519} (2000), 155--169.
  \MR{1739726 (2001k:11112)}

\bibitem{Tate}
John Tate, \emph{Endomorphisms of abelian varieties over finite fields},
  Invent. Math. \textbf{2} (1966), 134--144. \MR{0206004 (34 \#5829)}

\bibitem{Xia2}
Jie Xia, \emph{A characterization of the good reduction of {M}umford curve},
  arXiv: 1306.0163.

\bibitem{Xia4}
Jie Xia, \emph{Crystalline {H}odge cycles and {S}himura curves in positive
  characteristic}, arXiv:1311.0940.

\bibitem{Xia3}
Jie Xia, \emph{Tensor decomposition of isocrystals characterizes {M}umford
  curves}, arXiv:1310.2682.

\end{thebibliography}

\end{document}